\pdfoutput=1
\documentclass[11pt]{article}
\usepackage[a4paper, margin=1in,bottom=0.8in]{geometry}
\usepackage[colorlinks]{hyperref}
\usepackage{mathnotation}
\usepackage{thmtools}
\usepackage{thm-restate}
\usepackage{amssymb}
\usepackage{bbm}
\usepackage{mathtools}
\usepackage{authblk}
\usepackage{tikz}
\usetikzlibrary{math}
\usetikzlibrary{shapes.geometric}
\usetikzlibrary{backgrounds}
\usepackage{xfrac}
\usepackage{subcaption}
\newtheorem{theorem}{Theorem}[section]

\newtheorem{lemma}[theorem]{Lemma}
\newtheorem{claim}[theorem]{Claim}
\newtheorem{proposition}[theorem]{Proposition}

\theoremstyle{remark}

\theoremstyle{definition}
\newtheorem{problem}{Problem}

\begin{document}
\title{On the Connectivity and Diameter of Geodetic Graphs}
\author{Asaf Etgar, Nati Linial\thanks{Supported in part by grant 659/18 of the Israel Science Foundation}}
\affil{The Hebrew University of Jerusalem}
\maketitle
\begin{abstract}
    A graph $G$ is \emph{geodetic} if between any two vertices there exists a unique shortest path. In 1962 Ore raised the challenge to characterize geodetic graphs, but despite many attempts, such characterization still seems well beyond reach. We may assume, of course, that $G$ is $2$-connected, and here we consider only graphs with no vertices of degree $1$ or $2$. We prove that all such graphs are, in fact $3$-connected. We also construct an infinite family of such graphs of the largest known diameter, namely $5$.
\end{abstract}
\section{Introduction}
In this work we consider loopless, undirected graphs $G = (V,E)$. We think of a path in $G$ as a sequence of vertices $P = (v_1,\ldots v_k)$, and the subpath of $P$ from $v_j$ to $v_l$ is denoted $P(v_j,v_l)$. The length of a path $P$ is denoted $|P|$, and path concatenation is denoted by $*$. The distance between $u$ and $v$ is $d_G(v,u)=d(v,u)$, and is the length of a shortest $u,v$ path. For clarity, we occasionally add an index indicating the graph in which some parameter or quantity is calculated.

The notion of geodetic graphs was introduced by Ore \cite{ore1962theory} as a natural extension of trees: a tree is a graph in which between any two vertices there exists a unique simple path, and hence a unique shortest path (i.e - a geodesic). Ore purposed an extended definition, and asked in which simple graphs geodesics are unique. Some simple examples are trees, complete graphs, and odd-length cycles. 
Specifically, Ore raised the challenge to characterize geodetic graphs. Despite many attempts, a complete characterization still seems beyond reach.

There are easy necessary and sufficient  properties for a graph to be geodetic, the following can be easily proved:
\begin{claim}
A graph $G$ is not geodetic if and only if it contains an even circuit $C$ with two vertices $u,v\in C$ such that $d_C(u,v) = d_G(u,v) = \frac{|C|}{2}$.
\end{claim}
\begin{claim}\label{claim: blocks are geodetic}
A graph $G$ is geodetic if and only if each block of $G$ is geodetic.
\end{claim}
Here a \emph{block} is a maximal $2$-connected component of $G$. 
There is clearly no loss in generality if we restrict our attention to $2$-connected graphs.
By \autoref{claim: blocks are geodetic}, since a vertex of degree $1$ is a block, it suffices to assume all vertices have degree no less than $2$. Moreover, the following claim is given in \cite{stemple1968planar}:
\begin{claim}
    Let $G$ be a geodetic graph, and let $P = v_1,\ldots, v_k$ be a path with $\deg(v_1),\deg(v_k)\ge 3$ and $\forall i\ne 1,k\; \deg(v_i) = 2$. Then $P$ is the $v_1,v_k$ geodesic.
\end{claim}
By this claim, one can replace any path that consists of degree $2$ vertices by a single weighted edge. Therefore the question of geodeticity gains a more arithmetic flvaour.
Since our emphasis is combinatorial and geometric we will concentrate on graphs whose smallest degree is at least $3$.

There are not many families of geodetic graphs that are classified in full. The only constructive classifications known presently are planar geodetic graphs \cite{stemple1968planar}, which is extended to a classification of geodetic graphs homeomorphic to a complete graph \cite{stemple1979geodeticcompletegraph}. There is a considerable body of work classifying geodetic graphs of diameter $2$ \cite{stemple1974geodeticdiamtwo,scapellato1986geodeticdiamtwogeometric,blokhuis1988geodeticdiamtwo}, including a classification of such graphs. While some constructions are known, we do not know that they exhaust all possible geodetic graphs of diameter $2$. In this sense, the classification is lacking. Naturally - the question turns to higher diameters. Some properties of geodetic graphs of diameter $3$ are known \cite{parthasarathy1984geodeticdiamthree}. However, to the best of our knowledge the following is unknown:
\begin{problem}\label{diameter_three_existance}
Do there exist geodetic blocks $G$ of diameter $3$ with $\delta(G) \ge 3$? 
\end{problem}
Here $\delta(G)$ is the minimal degree of $G$. Progress on this problem has been very slow. Bridgland \cite{bridgland1983geodeticconvexity} constructed a family of geodetic blocks of diameter $4$ and arbitrarily large minimal degree. This construction was later generalized 
in several ways using block designs \cite{srinivasan1987construction}, 
yielding a family of geodetic blocks of diameter $5$. This construction has the largest diameter presently known.
Despite many attempts, we were unable to retrieve the latter paper. 
We therefore present these constructions along a different proof of their geodeticity.
A main problem that we raise is:

\begin{problem}\label{high_diameter_problem}
What is the largest possible diameter of a geodetic block with
minimal degree $\ge 3$? Can it be arbitrarily large?
\end{problem}

In the journey to classification, other properties of geodetic graphs were discovered. A graph is called\emph{ self centered} if its diameter equals its radius. Geodetic blocks of diameter $2$ are
known to have this property \cite{stemple1974geodeticdiamtwo}.  
Likewise, for blocks of diameter $3$ \cite{parthasarathy1984geodeticdiamthree}. Some connections to other graph properties were explored, namely by Zelinka \cite{zelinka1975geodetic}, Gorovoy and Zamiaikou \cite{gorovoy2021geodeticantipode}, and connections to other fields such as algebra and group theory \cite{elder2022rewriting,nebesky2002new}. We continue these lines of research, resulting in our main theorem:

\begin{restatable}{thm}{notwoconnected}\label{No2Connected}
Every $2$-connected geodetic graph $G$ with $\delta(G)\ge 3$ must be $3$-connected. This lower bound is tight as shown by the Petersen Graph.
\end{restatable}

\section{Geodetic Graphs and Connectivity}
In this section we prove \autoref{No2Connected}.
From here on we assume $G$ is geodetic with $\delta(G)\ge 3$. We denote the (unique) $v,u$ geodesic in $G$ by $\pi(v,u)$, and by convention we enumerate its vertices in order from $v$ to $u$. Arguing by contradiction, let $S=\{x,y\}$ be a vertex cut for which $d(x,y)$ is as small as possible, denote this distance by $\ell$. We denote by $\Pi=\pi(x,y)$ the $x,y$ geodesic, with vertices $\Pi= x, x_1, x_2, \ldots, x_{\ell-1}, y$. Let $A_1\ldots A_k$ the connected components of $G\setminus S$. Clearly, $x_1, x_2, \ldots, x_{\ell-1}$ all belong to the same connected component of $G\setminus S$, say they are in $A_1$.

\begin{lemma}\label{GeodeticComponents}
For every $i$, if $u,v\in A_i$, then $\pi_G(u,v)$ is contained in $A_i\cup \Pi$.
\end{lemma}
\begin{proof}
If $\pi_G(u,v)$ is not contained in $A_i\cup \Pi$, then it must leave $A_i$ and come back. But the only way to exit $A_i$ is via $x$ or $y$. But then $\pi_G(u,v) = \pi_G(u,x)*\Pi * \pi_G(y,v)$, since $\Pi$ is the $x,y$ geodesic. Thus $\pi_G(u,v)$ is contained in $A_i\cup \Pi$, as claimed.
\end{proof}

\begin{lemma}\label{Two Components}
The graph $G\setminus S$ has exactly two connected components, i.e., $k=2$.
\end{lemma}
\begin{proof}
Suppose toward contradiction that $A_1,A_2,A_3\neq\emptyset$. Let $\pi_i$ to be an $x,y$ geodesic in the subgraph induced by $A_i\cup S$. (in particular $\pi_1 = \Pi$). At least two of the integers $|\pi_1|, |\pi_2|, |\pi_3|$ have the same parity, so the corresponding paths form a cycle $C$ of even length. We consider two cases:
\begin{enumerate}
    \item $C = \pi_1*\pi_2^{-1}$: Since $\abs{\pi_2}\ge\abs{\pi_1}+2$ we can find two vertices $v_0,v_1\in A_2$ which are antipodal points on $C$. However, $\pi_2(v_0,v_1) = \pi_G(v_0,v_1)$ by \autoref{GeodeticComponents}. But $\pi_C(v_0,x)*\Pi*\pi_C(y,v_1)$ is another $v_0,v_1$ path of the same length, contradicting geodeticity.
    
    \item $C = \pi_2* \pi_3^{-1}$: Let $v_0,v_1$ be $C$ - antipodal points with $v_0\in A_2, v_1\in A_3$.
    The two arcs of $C$ that $v_0, v_1$ define are two $v_0,v_1$ paths of equal length. By assumption $G$ is geodetic so $\pi_G(v_0,v_1)$ differs from both these paths. But $\pi_G(v_0,v_1)$ must traverse either $x$ or $y$. This means e.g., that $|\pi_G(v_0,x)| < |\pi_2(v_0,x)|$ contrary to the assumption that  $\pi_2$ is a shortest $x,y$ path in $A_2\cup S$.
\end{enumerate} \end{proof}

A \emph{graph contraction} between two graphs $\Gamma, \Omega$, is a function $f:V(\Gamma)\to V(\Omega)$ such that for any $vu\in E(\Gamma)$, either $f(v)f(u)\in E(\Omega)$ or $f(v) = f(u)$. Clearly, for any $u,v\in V(\Gamma)$ it holds that $d_{\Omega}(f(v),f(u))\le d_\Gamma(v,u)$. Therefore if $\Gamma$ is connected, so is $f[\Gamma]$.\\

Let $\ZZ^2_\infty$ be the graph with vertex set $\ZZ^2$ where $p, q\in \ZZ^2$ are neighbors whenever $\norm{p-q}_{\infty} = 1$.
We denote coordinates in this plane by $(\xi,\eta)$ and employ this graph as a visualization tool
for $G$. This is accomplished using the
mapping $\varphi:V(G)\to \ZZ^2$, where 
\begin{align*}
\varphi(v) =\begin{cases} 
(d(x,v), d(y,v)) & v\in A_1\cup S \\
(-d(y,v) + \ell,  -d(x,v) + \ell) & v\in A_2 \cup S
\end{cases}     
\end{align*}
(Recall that $\ell = d(x,y)$).  We denote the image of $G$ in $\ZZ^2_\infty$
by $\Omega = \varphi[G]$. Clearly $\varphi$ is a graph contraction, and therefore $\varphi[A_i]$ is a connected subgraph of $\Omega$. For any $j$ define 
\[
R_j = \cbk{v\in A_1\mid d(v,y) - d(v,x) = j}\qquad L_j = \cbk{v\in A_2\mid d(v,y) - d(v,x) = j}
\]
For ease of notation, we add $x$ to $R_\ell,L_\ell$ and $y$ to $R_{-\ell}$, $L_{-\ell}$. We denote $R = \bigcup_{j=-\ell}^\ell R_j$ and $L = \bigcup_{j=-\ell}^\ell L_j$. By  \autoref{Two Components} $G = R\cup L$, and $\varphi[L_j],\varphi[R_j]$ are included in the straight line $\{(\xi,\xi + j)|\xi\in \ZZ\}$.

\begin{figure}[H]
\centering
\begin{subfigure}{0.43\textwidth}
    \includegraphics{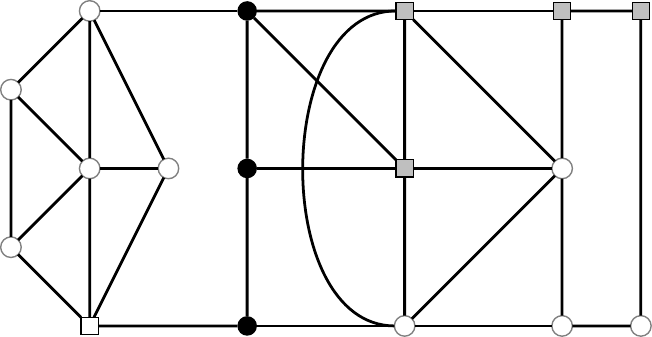}
\end{subfigure} 
\begin{subfigure}{0.45\textwidth}
    \includegraphics{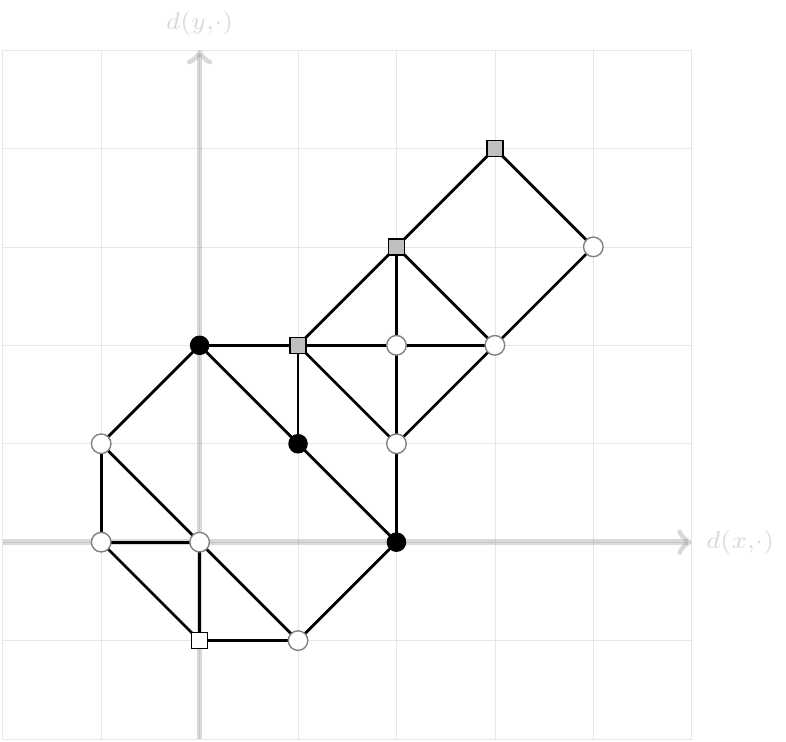}
\end{subfigure}

\caption{How $\varphi$ maps a $2$-connected graph. $\Pi = (x,x_1,y)$ are the black vertices, $R_1$'s vertices are the gray square vertices and $L_{-1}$ is the white square vertex.}    \end{figure}

In using $\varphi(G)$ as a visualization tool, we keep in mind that $\varphi$ is not injective. We note that $\varphi[\Pi]$ is the interval between $(0,\ell)$ and $(\ell,0)$, and that by geodeticity $\varphi^{-1}((j,\ell-j))$ is the $j$-th vertex in $\Pi$. In drawing $\varphi[R]$ and $\varphi[L]$, we note that $\varphi[R]$ is ``to the right" of $\varphi[\Pi]$, and $\varphi[L]$ is ``to the left" of $\varphi[\Pi]$. 
\begin{lemma}\label{NeighborsInPhi}
If $u,v$ are neighbors in $G$, $u\in R_j, v\in R_k$, then $\abs{k-j}\le 2$. Moreover, if $\abs{k-j} = 2$,
then the edge $\varphi(v)\varphi(u)$ is one of the two $ (\xi,\xi+j)\sim (\xi \pm1,\xi +j\mp 1)$.
\end{lemma}
\begin{proof}
A simple application of the triangle inequality.
\end{proof}
\begin{lemma}\label{lemma: existance of SE edges}
Suppose $R_j, R_{j-2} \neq \emptyset$ whereas $R_{j-1} = \emptyset$, then there exists an edge between $R_j$ and $ R_{j-2}$.
\end{lemma}
\begin{proof}
Since $A_1$ is connected, there must be a path between $R_j$ and $R_{j-2}$. Consider a shortest such path. It must completely reside in $R_j\cup R_{j-1}\cup R_{j-2}$, and since $R_{j-1}=\emptyset$, its first step outside of $R_j$, must be to a vertex in $R_{j-2}$, as claimed.
\end{proof}
The previous two lemmas clearly apply to $L$ as well.

\begin{lemma}\label{No Antipodal Distance difference}
If $\abs{j}< \ell$, then either $R_j$ or $L_{-j}$ must be empty.
\end{lemma}
\begin{proof}
We show that if there exist vertices $v\in R_j$, $u\in L_{-j}$, then $G$ is not geodetic, because $d(u,v)$ is realized by two distinct paths. Any $u,v$ path must clearly traverse either $x$ or $y$. But the assumption that $\abs{j} < \ell$ implies that the shortest $u,x,v$ path cannot traverse $y$
and the shortest $u,y,v$ path cannot traverse $x$. In particular,
the shortest $u,x,v$ path and $u,y,v$ path in $G$ are distinct. Moreover, they have the same length, because the shortest length of a $u,y,v$ resp.\ $u,x,v$ paths is $d(u,y) + d(y,v)$  resp.\ $d(u,x) + d(x,v)$. Since $d(u,y) - d(u,x) = j = d(v,x) - d(v,y)$, they have the same length, as claimed.
\end{proof}
It follows from \autoref{No Antipodal Distance difference} that at least one of $R_0, L_0$ must be empty. 
We denote below $\Pi= x, x_1, x_2, \ldots, x_{\ell-1}, y$. 
In particular, if $\ell=1$, then $x_1=y$. We need the following lemma.
\begin{lemma}
\label{Find square}
If both $R_{\ell-1}=R_{\ell-3}=\emptyset$, then $R_\ell= \{x\}$. In particular, $x_{1}$ is the only neighbor of $x$ in $A_1\cup\Pi$.  
\end{lemma}
\begin{proof}
The vertex $x_1$ has degree at least $3$, and must therefore have a neighbor in $R$.
But by assumption $R_{\ell-1} = R_{\ell-3} = \emptyset$, so $x_1$ must have a neighbor 
in $R_{\ell - 2}$. Consequently, $x_1$ is not the only vertex in $R_{\ell - 2}$.
The graph $\varphi[A_1]$ is connected, and since $\varphi^{-1}(0,\ell)=\{x\}$, it does not contain the vertex 
$(0,\ell)$, nor the edge $(0,\ell)\sim (1,\ell-1)$. Let us assume toward contradiction that $R_\ell$ is comprised not only of $x$. Then $\varphi[R_\ell]\setminus\{(0,\ell)\}$ and $R_{\ell -2}$, are nonempty whereas $R_{\ell-1} = \emptyset$. By \autoref{lemma: existance of SE edges}, there exists an edge $(\xi,\xi+\ell)\sim (\xi+1,\xi+\ell - 1)$. Let $v_\xi u_\xi\in E(G)$ be a pre-image of this edge. Namely, $\varphi(v_\xi) = (\xi,\xi+\ell), \varphi(u_\xi) = (\xi + 1, \xi + \ell -1)$. Clearly $\pi(v_\xi,x)$ must be contained in $R_\ell$. Moreover, $\pi(u_\xi,y)$ does not go through $x$ because $d(u_\xi,y) - d(u_\xi,x) = \ell-2$. Therefore, $\pi(v_\xi,x)*\Pi$ and $v_\xi * \pi(u_\xi,y)$ are two distinct $v_\xi,y$ geodesics - contrary to the assumption of geodeticity.

\end{proof}
A similar lemma can be proved for $R_{-\ell}$: If both $R_{1 - \ell} = R_{3 - \ell} = \emptyset$, then $R_{-\ell} = \{y\}$. In particular, $x_{\ell-1}$ is the only neighbor of $y$ in $A_1\cup\Pi$.  
\begin{figure}[H]
    \centering
     \includegraphics[scale=.86]{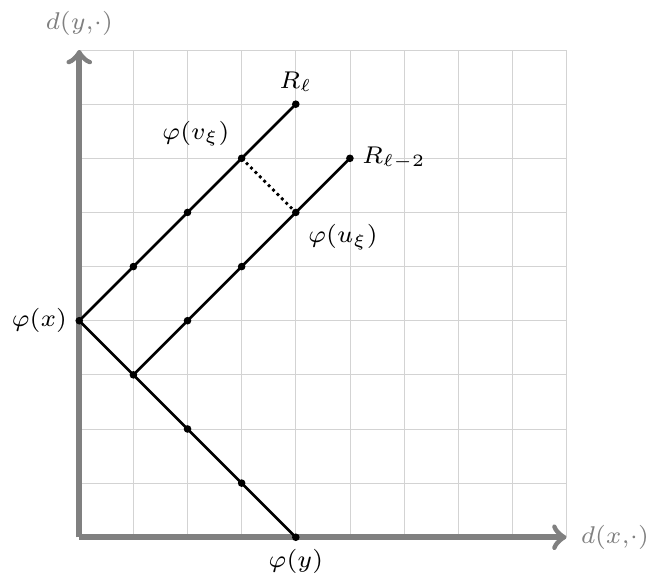}
    \caption{\centering Illustration of \autoref{Find square}. The existence of the dotted edge follows from the assumption that $x$ is not the only vertex of $R_\ell$.}
    \label{fig: visual of south east edge}
\end{figure}

We can now prove \autoref{No2Connected}.
\notwoconnected*
\begin{proof}
The argument runs as follows: The sets $R_j$ are nonempty for every other value of $j$. To wit, $R_{\ell - 2j}$ is nonempty for any $1\le j\le \ell-1$, since $R_{\ell - 2j}$ contains the vertex $x_j$. But then \autoref{No Antipodal Distance difference} implies that $L_{\ell-2j}= \emptyset$, see \autoref{fig:Stage 2}. We claim that $L_{\ell - 1}$ and $L_{1-\ell}$ are nonempty, thus repeated application of \autoref{lemma: existance of SE edges} results in $L_{k}\neq\emptyset$ for any $k\not\equiv \ell \mod 2$. Therefore, $R_k = \emptyset$ for such $k$. Now we satisfy the assumptions of \autoref{Find square}, so $R_\ell = \{x\}$. But then $\{x_1,y\}$ is a vertex cut. If $xy\in E(G)$, then $x_1 = y$, contrary to $G$ being $2$-connected. Otherwise,  $d(x_1,y)<d(x,y)$ - contrary to the minimality of $\ell$.

It is left to justify that $L_{\ell-1},L_{1-\ell}$ are nonempty. The neighbors of
$x$ in $A_2$ reside in $L_\ell, L_{\ell-1}$. If $x$ does not have a neighbor in $L_\ell$, we are done, so suppose $x$ has a neighbor in $L_\ell$, and $y$ has a neighbor $v_y\in L_{-\ell}\cup L_{1-\ell}$. Since $A_2$ is connected, there must be a path connecting these vertices. Consider a shortest such path, and its first step outside of $L_{\ell}$. It cannot be a vertex in $L_{\ell-2}$, so it must be in $L_{\ell-1}$.
\end{proof}
\begin{figure}[H]
\centering
    \begin{subfigure}[b]{0.45\textwidth}
    \centering
    \includegraphics{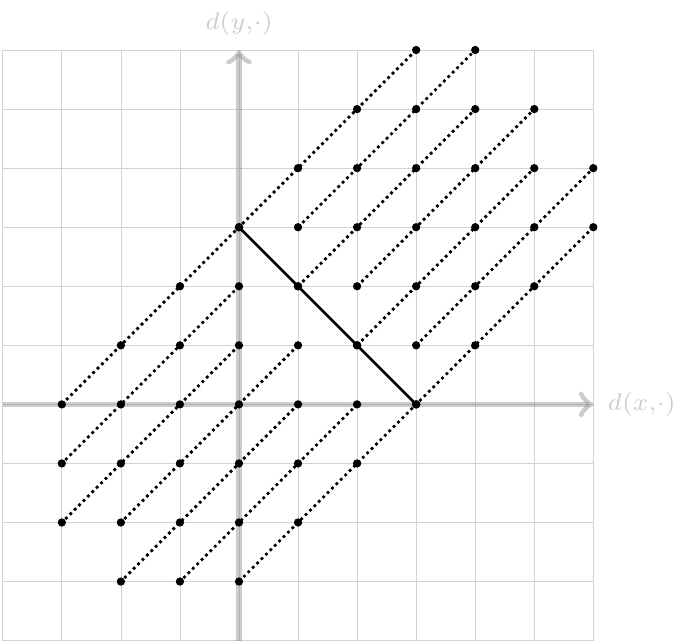}
    \caption{Stage 1}
    \end{subfigure}
    \begin{subfigure}[b]{0.45\textwidth}
    \centering
    \includegraphics{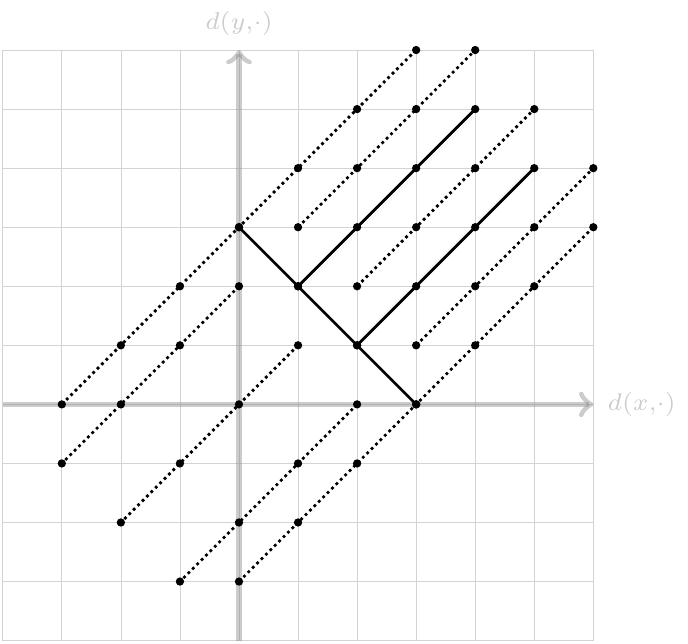}
    \caption{Stage 2: $R_{\ell - 2j}\neq \emptyset$}
    \label{fig:Stage 2}
    \end{subfigure}
    \begin{subfigure}[b]{0.45\textwidth}
    \centering
    \includegraphics{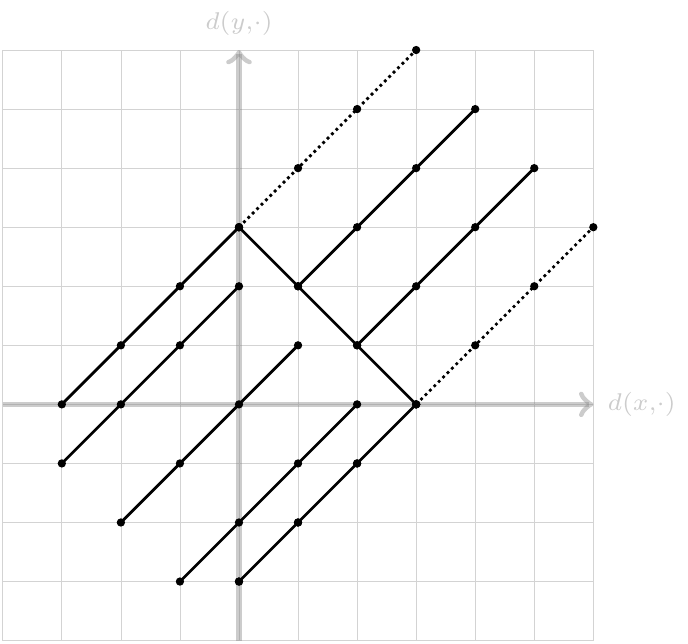}
    \caption{Stage 3: $L_k \neq \emptyset$ for $k
    \not\equiv \ell \mod 2$}
    \label{fig:Stage 3}
    \end{subfigure}
    \begin{subfigure}[b]{0.45\textwidth}
    \centering
    \includegraphics{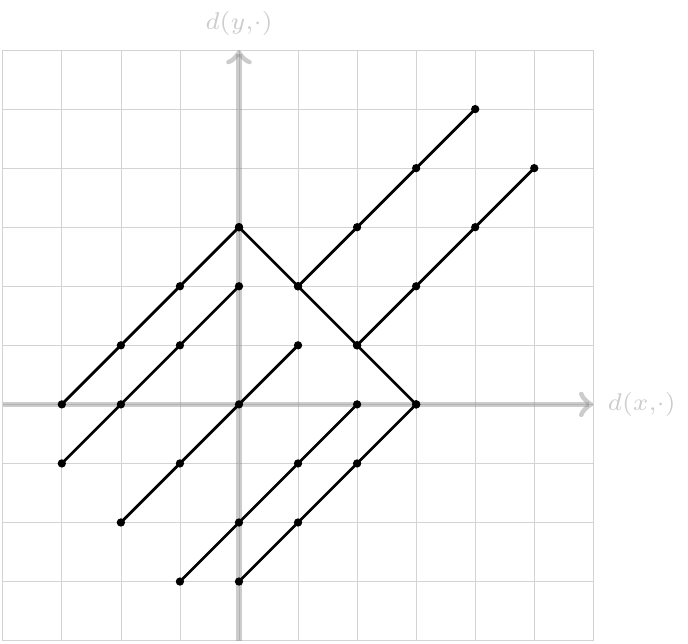}
    \caption{Stage 4: $R_\ell = \{x\}$}
    \label{fig:Stage 4}
    \end{subfigure}
    \caption{Illustration of the proof: Black lines represent diagonals which are known to be nonempty. Dotted lines stand for the presently undecided cases. A diagonal whose status is decided becomes black if proven nonempty, and deleted if empty.}
\end{figure}

\section{Geodetic blocks of diameter 4 and 5 \label{point_flag_section}}
In this section, we construct two families of geodetic blocks, the graphs in which
have diameter $4$ and $5$. We need some notation first. Let $v$ be a vertex in a
graph $G = (V,E)$ and let $i\ge 0$ be an integer. Clearly $|d(v,x)-d(v,y)|\le 1$ for every edge $e=xy$ in $G$. 
We say that $e$ is {\em $v$-horizontal} resp.\ {\em $v$-vertical} if $d(v,x)=d(v,y)$ resp.
$|d(v,x)-d(v,y)|= 1$. Note that every edge in every shortest $v\to u$ path is $v$-vertical.

\begin{proposition}
$G$ is geodetic if and only if for every $v\in V$ the $v$-vertical edges form a spanning tree.
\end{proposition}
\begin{proof}
It is easy to see that the $v$-vertical edges form a spanning subgraph in every connected graph.
Suppose that $G$ is not geodetic, and let us find some vertex $w$ such that the $w$-vertical 
edges in $G$ do not form a spanning tree, as there is a cycle comprised of $w$-vertical edges. Indeed,  
since $G$ is not geodetic, it has vertices $v,u$ with two distinct shortest paths between them $\pi\neq\pi'$.
But all the edges in $\pi,\pi'$ are $v$-vertical, and together they contain a cycle, as claimed.

Conversely let $v$ be a vertex in a geodetic graph $G$,
and let $T$ be a BFS tree rooted at $v$. Clearly all edges in $T$ are
$v$-vertical, and as we show, all edges $xy\notin T$ are $v$-horizontal. Indeed, 
if $xy$ is $v$-vertical, with $d(v,x) = i$ and $d(v,y) = i+1$, 
this yields two distinct shortest paths from $v$ to $y$.
\end{proof}

We mostly follow the notation and terminology of \cite{van2001course}.
Let $q$ be a prime power, and
let $PG_2(q)$ and $AG_2(q)$ be a projective resp.\ affine plane of order $q$. 
The point sets of these geometries is denoted by $P$, and
their sets of lines (blocks) by $\Ll$. Their point-line incidence relation
is denoted by $\Ii \subset P\times \Ll$. 
Elements of $\Ii$ are called \emph{Flags}. The \emph{Levi Graph} of an incidence structure $\SS = (P,\Ll,\Ii)$, 
denoted $Levi(\SS)$ is the bipartite graph with vertex sets $P\sqcup \Ll$, 
where $p$ and $L$ are neighbors iff $p$ is incident with $L$.

The \emph{Flag graph} of $\SS$ is denoted $Flag(\SS)$. Its vertex set is $P\sqcup \Ii$.
It has two kinds of edges: between a point $p$ and flag $(p,L)$. In addition $(p,L)\sim (p',L)$
for every two points $p,p'$ of the same line $L$. As we show below
$Flag(\SS)$ is geodetic, and has diameter $4$ when $\SS = PG_2(q)$ and $5$ if $\SS = AG_2(q)$. 
To fix ideas, associated with each $L\in \Ll$
in $Levi(\SS)$ is a {\em porcupine}, a clique of size $|L|$, plus an
edge $p\sim(p,L)$ emanating from $(p,L)$ for every $p\in L$.
Therefore, to every simple path $Q$ in $Levi(\SS)$ 
there corresponds a simple path $\hat{Q}$ in $Flag(\SS)$. Namely,
if $Q$ traverses through $L$, that is $[p,L,p']$, the corresponding steps in $\hat{Q}$
$[p,(p,L),(p',L),p']$, which is a simple path in $Flag(\SS)$. Recall
that a graph is $2$-connected if and only if every two of its vertices lie on a simple cycle. We conclude:

\begin{cor}\label{levi:cohen}
If $Levi(\SS)$ is $2$-connected, then so is $Flag(\SS)$.
\end{cor}
\subsection{Properties of $Flag(AG_2(q))$}
Recall the following properties of $AG_2(q)$:
\begin{enumerate}
    \item \label{numofitems} It has $q^2$ points and $q^2 + q$ lines.
    \item Every point is incident with $q+1$ lines
    \item \label{pointsinline}Every line has $q$ points
    \item If a point $p$ is not in a line $L$, then there exists a unique line $L'$ with $p\in L'$ such that $L$ and $L'$ are disjoint. The common practice is to say that $L$ and $L'$ are parallel and denote $L\parallel L'$.
\end{enumerate}
We denote $\mathscr{F}(q) = Flag(AG_2(q))$. Properties \ref{numofitems} and \ref{pointsinline} imply that the number of vertices in $\mathscr{F}(q))$ is $q^3 + 2q^2$. 
We denote by $L_{\alpha,\beta}$ the unique line that contains the two points $\alpha,\beta$.

\begin{proposition}
\label{prop: flag_graph_2_connected}
 $\mathscr{F}(q)$ is 2-connected.
\end{proposition}
\begin{proof}
We exhibit a simple cycle through any two vertices in $AG_2(q)$.
Consider all cycles of the form
\[(L_1,x_1,L_{x_1,x_2},x_2,L_2,x,L_1)\]
where $L_1, L_2$ are two distinct intersecting lines with $L_1\cap L_2 = x$, and
$x_1,x_2$ are points in $L_1,L_2$ respectively, other than $x$. 
These yield cycles through any two vertices in $\mathscr{F}(q)$ other than a pair of parallel lines. Finally,
here is a simple cycle through two parallel lines $L_1, L_2$
\[(L_1,x_1,L_{x_1,x_2},x_2,L_2,x'_2,L_{x_1',x_2'},x_1',L_1).\]
Here $x_1,x_1', x_2,x_2'$ are distinct points on $L_1, L_2$. This completes the proof.
\end{proof}

\begin{theorem}
\label{thm: affine_flag_geodetic}
$\mathscr{F}(q)$ is geodetic and of diameter $5$.
\end{theorem}
\begin{proof}
We show that the collection of $v$-vertical edges form a tree for every vertex $v$ in $\mathscr{F}(q)$.
There are just two cases to consider: $v = \bf{p}$, a point in $P$ or a flag $v = \bf{(p,L)}\in \Ii$.
Let $\Vv(i)$ be the number
of $v$-vertical edges $\alpha\beta\in E$ of {\em height} $i$, i.e.,
$\{(d(v,\alpha),d(v,\beta)\} = \{i-1,i\}$. Let $N_i(v)$ be the $i$-sphere centered at $v$, that is $N_i(v) = \cbk{u\in V\mid d(v,u) = i}$. We analyze $N_i(v)$ in either case.
\paragraph{$\mathbf{v = p}$:} 

\begin{enumerate}
    \item Clearly $N_1(\bf{p}) = \cbk{(\bf{p},L)\in \Ii}$, so $\Vv(1) = |N_1(\bf{p})| = q+1$.
    \item Each $(\bf{p},L)\in N_1(\bf{p})$ is adjacent to all the vertices in the clique defined by $L$.
    This contributes $q-1$ vertical edges. Also, $N_2(\bf{p}) = \cbk{(x,L)\mid \bf{p}\in L, x\neq p}$, since
    these cliques are disjoint. It also follows that $\Vv(2) = |N_1(p)|\cdot (q-1) = (q^2-1)$.
    \item Let $(x,L)\in N_2(\bf{p})$. The neighbors of $(x,L)$ are $x$ and $(y,L)$ for $y\neq x$ in $L$.
    All the latter are in $N_2(\bf{p})$, so every $(x,L)\in N_2(\bf{p})$ contributes exactly one vertex
    to $N_3(\bf{p})$. Consequently, $N_3(v) = \cbk{x\mid x\ne \bf{p}}$, and $\Vv(3) = \Vv(2) = (q^2-1)$.
    \item Each $x\ne \bf{p}$ is incident with $q+1$ lines, only one of which is $L_{\bf{p},x}$. Therefore, $x$ is adjacent
    to $q$ flags $(x,L')$, one for each  $L'\ne L_{\bf{p},x}$. These flags are clearly distinct, therefore $N_4(v) = \cbk{(x,L')\mid L'\ne L_{\bf{p},x}}$ and $\Vv(4) = q\cdot \Vv(3) = q(q^2-1)$.
\end{enumerate}
The calculation checks:
\[
    \overbrace{(q+1)}^{\Vv(1)} + \overbrace{(q^2-1)}^{\Vv(2)} + \overbrace{(q^2-1)}^{\Vv(3)} + \overbrace{q(q^2-1)}^{\Vv(4)} = q^3+2q^2-1 = |V(\mathscr{F}(q))| - 1.
\]
\begin{figure}[H]
    \centering
    \includegraphics[scale = .3]{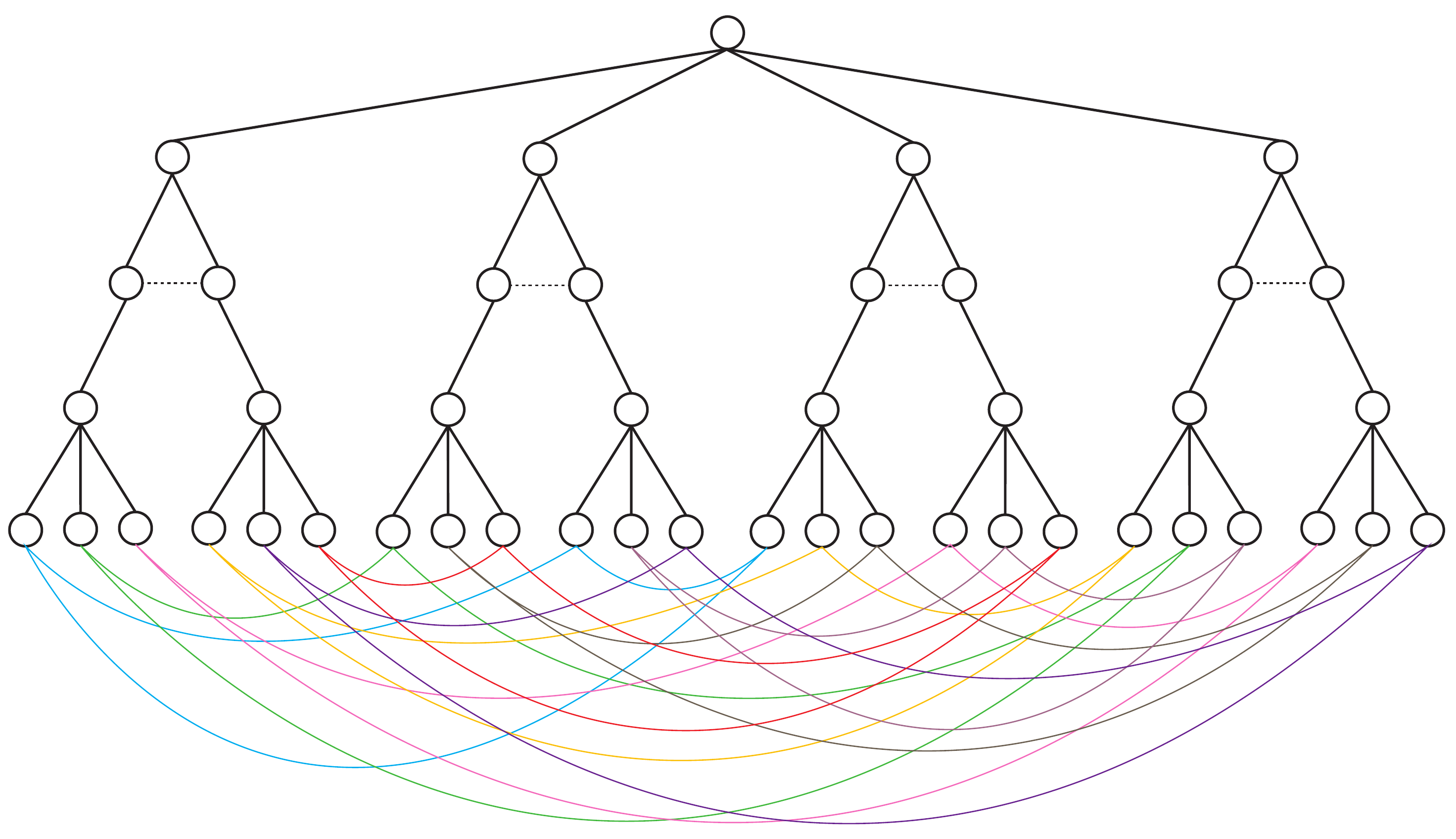}
    \caption{$\mathscr{F}(3)$, as seen from $v=\bf{p}$. The colors have no mathematical significance and are only intended for better visibility.}
    \label{fig:my_label}
\end{figure}

\paragraph{$\mathbf{v = (p,L)}$:} This case is a bit trickier. Let $x\neq\mathbf{p}$ be a point on $\mathbf{L}$.
Let $y$ be a point not on $\mathbf{L}$, and $\ell_y$ the unique line through $y$
that is parallel to $\mathbf{L}$. 
Finally, we denote by $\cbk{\bf{L}^i_x}_{i=0}^q$ the lines through $x$, where $\bf{L} = \bf{L}_x^0$.
\begin{enumerate}
    \item  $N_1(v) = \{\bf{p}\} \sqcup \{(x,\bf{L})\mid x\neq \bf{p}\}$, so $\Vv(1) = 1 + (q-1) = q$.
    We refer below to descendants of $\bf{p}$ as the left vertices, and to descendants of $\{(x,\bf{L})\mid x\neq \bf{p}\}$ as right vertices (see \autoref{fig: second_case}).
    
    \item The neighbors of $\bf{p}$ other than $\bf{(p,L)}$ are $ \cbk{(\bf{p},\bf{L}_\bf{p}^i)}_{i\ne 0}$, so the left side in $N_1(v)$ contributes $q$ edges of height $2$. The right vertices in $N_1(v)$ form a clique, with a porcupine structure. Therefore, each vertex of this clique contributes a single height $2$ edge, namely $(x,\bf{L})\sim x$. So $N_2(\bf{(p,L)}) = \cbk{(\bf{p},\bf{L}^i_{\bf{p}})}_{i\ne 0}\sqcup \cbk{x\mid \bf{p}\ne x\in \bf{L}}$ and $\Vv(2) = q + (q-1)$.
    \item Each left vertex in $N_2(v)$ is adjacent to a clique $\cbk{(y,\bf{L}^i_\bf{p})}_{y\ne \bf{p}}$ of $q-1$ flags, so each $(\bf{p},\bf{L}^i_\bf{p})$ contributes $(q-1)$ distinct edges of height $3$. As for the right side, the edges $x\sim (x,\bf{L}^i_x)$ have height $3$, and each $x\ne \bf{p}$ contributes $q$ of them.\\ So, $N_3(v) = \cbk{(y,\bf{L}^i_\bf{p})\mid y\ne p, i\ne 0}\sqcup \cbk{(x,\bf{L}_x^i)\mid x\ne \bf{p}, i\ne 0}$ and $\Vv(3) = q(q-1) + q(q-1)$.
    \item For every $y \notin \bf{L}$, there holds $L_{\bf{p},y} = \bf{L}_\bf{p}^i$ for some $i\ne 0$. Therefore, the left side vertices of $N_3(v)$ are partitioned into cliques in $N_3(v)$, and cover all $y\notin \bf{L}$. Hence each vertex contributes a unique edge $(y,L_{\bf{p},y})\sim y$ of height $4$. As for the right side: Every edge $(x,\bf{L}_x^i)$ of height $4$ is adjacent to the clique $\cbk{(y,L_x^i)\mid y\ne x}$. There are exactly $q-1$ such vertices for each $(x,\bf{L}_x^i)$. So $N_4(v) = \cbk{y\mid y\notin \bf{L}} \sqcup \cbk{(y,\bf{L}_x^i)\mid x\in \bf{L},y\notin \bf{L}, i\in [q]}$ and $\Vv(4) = q(q-1) + q(q-1)^2$.
    \item The edges $(y,L_x^i)\sim y$ are horizontal (between the two sides of $N_4(v)$), so the only $5$-vertical edges are of the form $y\sim (y,\ell_y)$ - and since $\ell_y$ is unique, $\Vv(5) = q(q-1)$ and $N_5(v) = \cbk{(y,\ell_y)\mid y\notin \bf{L}}$.
\end{enumerate}
Once again, the calculation checks:
\[
\overbrace{1 + (q-1)}^{\Vv(1)} + 
\overbrace{q+(q-1)}^{\Vv(2)} + \overbrace{q(q-1)+q(q-1)}^{\Vv(3)} + \overbrace{q(q-1)+q(q-1)^2}^{\Vv(4)} + \overbrace{q(q-1)}^{\Vv(5)} = q^3 + 2q^2 - 1.
\]
This analysis also establishes that $\mathscr{F}(q)$ has diameter 5.
such edges. Each $x\in L$ lies on $q$ lines other than $L$, so each $x$ contributes $p$ additional $3$-vertical edges of the form $x(x,L')$ for a total of $q(q-1)$ such edges. Each such $L'$ contains $p-1$ points other than $x$, and said points d not lie on $L$. Thus each $(x,L')$ contributes $p-1$ vertical edges of the form $(x,L')(y,L')$, for a total of $q\cdot (q-1)^2$ edges. Neighbors of $(y,L')$ are the points $y$, which we are of distance $4$ from $(p,L)$ - as we see in the fllowing part.
\end{proof}
\begin{figure}[H]
    \centering
    \includegraphics[scale=.3]{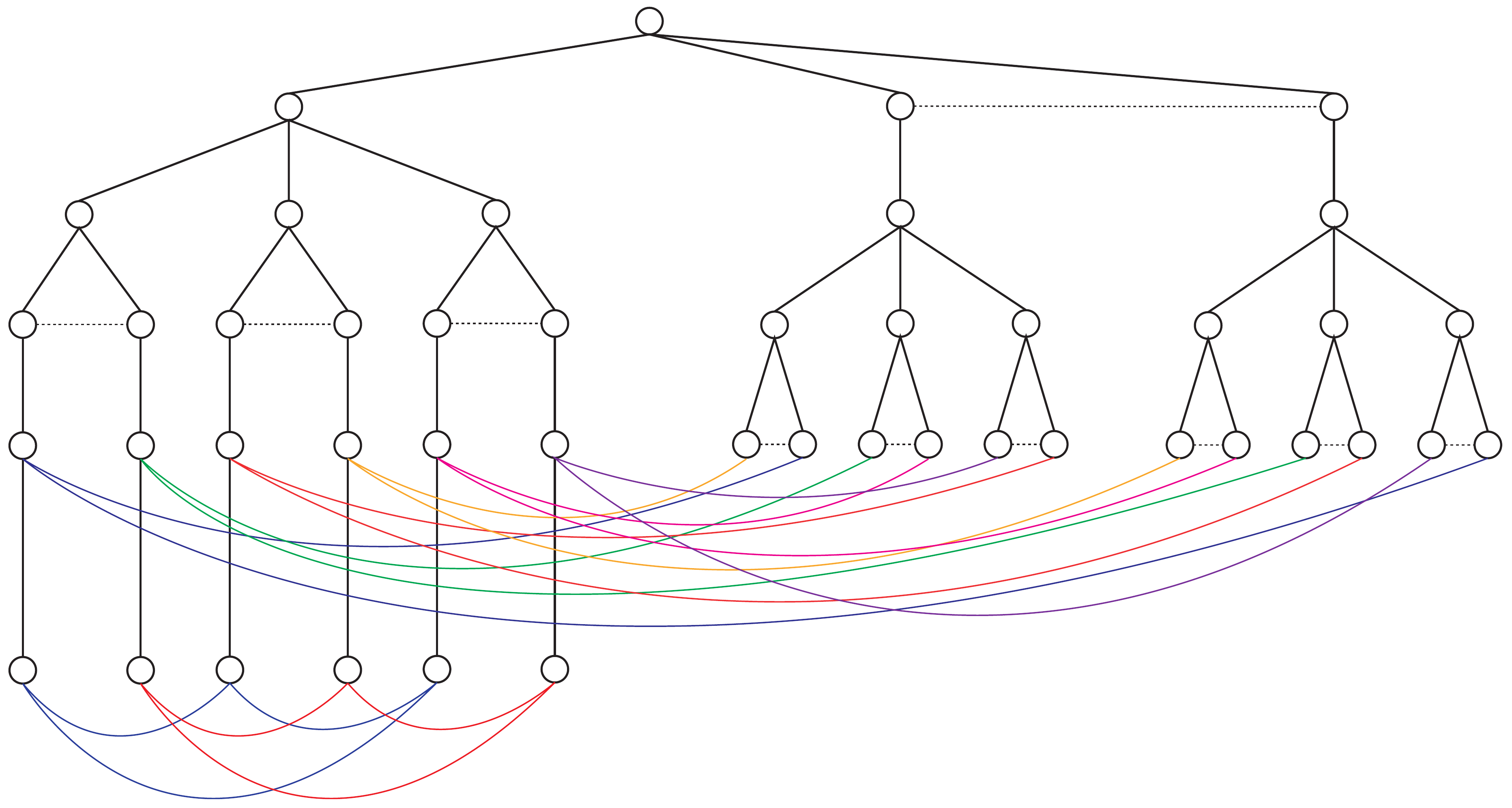}
    \caption{$\mathscr{F}(3)$, as seen from $v = \bf{(p,L)}$.}
    \label{fig: second_case}
\end{figure}

\subsection{Properties of $Flag(PG_2(q))$}
Recall the following properties of $PG_2(q)$:
\begin{enumerate}
    \item It has $q^2 + q + 1$ points and the same number of lines.
    \item Every point is incident with $q+1$ lines, and every line has $q+1$ points.
\end{enumerate}
These properties imply that $\abs{V(Flag(PG_2(q))} = (q+1)^3+1$.
\begin{theorem}
$Flag(PG_2(q))$ is $2$-connected, geodetic and has diameter $4$.
\end{theorem}
\begin{proof}
By \autoref{levi:cohen}, if $Levi(PG_2(q))$ is $2$-connected, then so is
$Flag(PG_2(q))$. To show that $Levi(PG_2(q))$ is $2$-connected, it suffices to use
the first cycle that is described in the proof of \autoref{prop: flag_graph_2_connected}. 
The proof of geodeticity is a slight adaptation of the proof of \autoref{thm: affine_flag_geodetic}.
The numbers change somewhat, and what's more, since no lines are parallel, we do not reach $N_5(v)$ 
in the second case. Consequently, the diameter is $4$.
\end{proof}

\section{Discussion}
Geodetic blocks with $\delta(G)\ge 3$ 
(which, by Theorem \ref{No2Connected}, are $3$ connected) seem hard to find. 
Specifically, we ask:
\begin{enumerate}
    \item We recall our question whether geodetic blocks of $\delta(G)\ge 3$ can
    have arbitrarily large diameter. Also, can they have arbitrarily large girth?
    The two questions are closely related since the diameter of a graph is at least half its girth. 
    \item So far, we only know of three cubic geodetic blocks - the Petersen Graph, $K_4$ and $Flag(PG_2(2))$. Is this list exhaustive? Is the number of such graphs finite?
    \item Is the study of geodetic graphs related to structural graph theory, and more
    concretely to the family of \emph{even-hole-free graphs} \cite{vuvskovic2010even}? The
    origin of this question is this: If $u,v$ are two antipodal vertices in an induced even cycle
    (aka an even hole) $C$ in a geodetic graph, then the $uv$ geodetic is not included in $C$.
\end{enumerate}

\bibliography{refs}

\providecommand{\bysame}{\leavevmode\hbox to3em{\hrulefill}\thinspace}
\providecommand{\MR}{\relax\ifhmode\unskip\space\fi MR }
\providecommand{\MRhref}[2]{%
  \href{http://www.ams.org/mathscinet-getitem?mr=#1}{#2}
}
\providecommand{\href}[2]{#2}
\begin{thebibliography}{10}

\bibitem{blokhuis1988geodeticdiamtwo}
A~Blokhuis and AE~Brouwer, \emph{Geodetic graphs of diameter two}, Geometries
  and Groups, Springer, 1988, pp.~527--533.

\bibitem{bridgland1983geodeticconvexity}
Michael~Franklyn Bridgland, \emph{Geodetic graphs and convexity}, Ph.D. thesis,
  Louisiana State University and Agricultural \& Mechanical College, 1983.

\bibitem{elder2022rewriting}
Murray Elder and Adam Piggott, \emph{Rewriting systems, plain groups, and
  geodetic graphs}, Theoretical Computer Science \textbf{903} (2022), 134--144.

\bibitem{gorovoy2021geodeticantipode}
Dmitriy Gorovoy and David Zmiaikou, \emph{On graphs with unique geodesics and
  antipodes}, arXiv preprint arXiv:2111.09987 (2021).

\bibitem{nebesky2002new}
Ladislav Nebesk{\`y}, \emph{New proof of a characterization of geodetic
  graphs}, Czechoslovak Mathematical Journal \textbf{52} (2002), no.~1, 33--39.

\bibitem{ore1962theory}
Oystein Ore, \emph{Theory of graphs}, American Mathematical Society, 1962.

\bibitem{parthasarathy1984geodeticdiamthree}
KR~Parthasarathy and N~Srinivasan, \emph{Geodetic blocks of diameter three},
  Combinatorica \textbf{4} (1984), no.~2, 197--206.

\bibitem{scapellato1986geodeticdiamtwogeometric}
Raffaele Scapellato, \emph{Geodetic graphs of diameter two and some related
  structures}, Journal of Combinatorial Theory, Series B \textbf{41} (1986),
  no.~2, 218--229.

\bibitem{srinivasan1987construction}
N~Srinivasan, J~Opatrny, and VS~Alagar, \emph{Construction of geodetic and
  bigeodetic blocks of connectivity k-greater-than-or-equal-to-3 and their
  relation to block-designs}, Ars Combinatoria \textbf{24} (1987), 101--114.

\bibitem{stemple1974geodeticdiamtwo}
Joel~G Stemple, \emph{Geodetic graphs of diameter two}, Journal of
  Combinatorial Theory, Series B \textbf{17} (1974), no.~3, 266--280.

\bibitem{stemple1979geodeticcompletegraph}
Joel~G {Stemple}, \emph{Geodetic graphs homeomorphic to a complete graph},
  Annals of the New York Academy of Sciences \textbf{319} (1979), no.~1,
  512--517.

\bibitem{stemple1968planar}
Joel~G Stemple and Mark~E Watkins, \emph{On planar geodetic graphs}, Journal of
  Combinatorial Theory \textbf{4} (1968), no.~2, 101--117.

\bibitem{van2001course}
Jacobus~Hendricus van Lint and Richard~Michael Wilson, \emph{A course in
  combinatorics}, Cambridge university press, 2001.

\bibitem{vuvskovic2010even}
Kristina Vu{\v{s}}kovi{\'c}, \emph{Even-hole-free graphs: a survey}, Applicable
  Analysis and Discrete Mathematics (2010), 219--240.

\bibitem{zelinka1975geodetic}
Bohdan Zelinka, \emph{Geodetic graphs of diameter two}, Czechoslovak
  Mathematical Journal \textbf{25} (1975), no.~1, 148--153.

\end{thebibliography}
\bibliographystyle{amsplain}
\end{document}